\documentclass[oneside,english]{amsart}
\usepackage{lmodern}
\usepackage[T1]{fontenc}
\usepackage[latin9]{inputenc}
\usepackage{geometry}
\usepackage{mathrsfs}
\usepackage{amsbsy}
\usepackage{amstext}
\usepackage{amsthm}
\usepackage{amssymb}

\makeatletter
\numberwithin{equation}{section}
\numberwithin{figure}{section}
\theoremstyle{plain}
\newtheorem{thm}{\protect\theoremname}
\theoremstyle{plain}
\newtheorem{prop}[thm]{\protect\propositionname}
\theoremstyle{plain}
\newtheorem{lem}[thm]{\protect\lemmaname}
\theoremstyle{remark}
\newtheorem{rem}[thm]{\protect\remarkname}

\makeatother

\usepackage{babel}
\providecommand{\lemmaname}{Lemma}
\providecommand{\propositionname}{Proposition}
\providecommand{\remarkname}{Remark}
\providecommand{\theoremname}{Theorem}

\begin{document}
\global\long\def\e{e}%
\global\long\def\V{{\rm Vol}}%
\global\long\def\bs{\boldsymbol{\sigma}}%
\global\long\def\br{\boldsymbol{\rho}}%
\global\long\def\bp{\boldsymbol{\pi}}%
\global\long\def\btau{\boldsymbol{\tau}}%
\global\long\def\bx{\mathbf{x}}%
\global\long\def\by{\mathbf{y}}%
\global\long\def\bz{\mathbf{z}}%
\global\long\def\bv{\mathbf{v}}%
\global\long\def\bu{\mathbf{u}}%
\global\long\def\bi{\mathbf{i}}%
\global\long\def\bn{\mathbf{n}}%
\global\long\def\grad{\nabla_{sp}}%
\global\long\def\Hess{\nabla_{sp}^{2}}%
\global\long\def\lp{\Delta_{sp}}%
\global\long\def\gradE{\nabla_{\text{Euc}}}%
\global\long\def\HessE{\nabla_{\text{Euc}}^{2}}%
\global\long\def\HessEN{\hat{\nabla}_{\text{Euc}}^{2}}%
\global\long\def\ddq{\frac{d}{dR}}%
\global\long\def\qs{q_{\star}}%
\global\long\def\qss{q_{\star\star}}%
\global\long\def\lm{\lambda_{min}}%
\global\long\def\Es{E_{\star}}%
\global\long\def\As{A_{\star}}%
\global\long\def\EH{E_{\Hess}}%
\global\long\def\Esh{\hat{E}_{\star}}%
\global\long\def\ds{d_{\star}}%
\global\long\def\Cs{\mathscr{C}_{\star}}%
\global\long\def\nh{\boldsymbol{\hat{\mathbf{n}}}}%
\global\long\def\BN{\mathbb{B}^{N}}%
\global\long\def\ii{\mathbf{i}}%
\global\long\def\SN{\mathbb{S}^{N-1}}%
\global\long\def\SM{\mathbb{S}^{M-1}}%
\global\long\def\SNq{\mathbb{S}^{N-1}(q)}%
\global\long\def\SNqd{\mathbb{S}^{N-1}(q_{d})}%
\global\long\def\SNqp{\mathbb{S}^{N-1}(q_{P})}%
\global\long\def\nd{\nu^{(\delta)}}%
\global\long\def\nz{\nu^{(0)}}%
\global\long\def\cls{c_{LS}}%
\global\long\def\qls{q_{LS}}%
\global\long\def\dls{\delta_{LS}}%
\global\long\def\E{\mathbb{E}}%
\global\long\def\P{\mathbb{P}}%
\global\long\def\R{\mathbb{R}}%
\global\long\def\spp{{\rm Supp}(\mu_{P})}%
\global\long\def\indic{\mathbf{1}}%
\global\long\def\lsc{\mu_{{\rm sc}}}%
\newcommand{\SNarg}[1]{\mathbb S^{N-1}(#1)}
\global\long\def\se{s(E)}%
\global\long\def\ses{s(\Es)}%
\global\long\def\so{s(0)}%
\global\long\def\sef{s(E_{f})}%
\global\long\def\seinf{s(E_{\infty})}%
\global\long\def\L{\mathcal{L}}%
\global\long\def\gflow#1#2{\varphi_{#2}(#1)}%
\global\long\def\S{\mathscr{S}}%
\global\long\def\Asm{{\rm (A)}}%
\global\long\def\Asmnb{{\rm A}}%
\global\long\def\Frep{F^{{\rm Rep}}}%
\global\long\def\s{\mathfrak{s}}%
\global\long\def\e{e}%
\global\long\def\EsN{E_{\star,N}}%

\theoremstyle{definition}
\newtheorem*{assumption}{Assumption}
\title{On the second moment method and RS phase of multi-species spherical
spin glasses}
\author{Eliran Subag}
\begin{abstract}
Excluding some special cases, computing the critical inverse-temperature
$\beta_{c}$ of a mixed $p$-spin spin glass model is a difficult
task. The only known method to calculate its value for a general model
requires the full power of the Parisi formula. On the other hand,
an 
easy application of the second moment method to the partition function
yields an explicit lower bound $\beta_{m}\leq\beta_{c}$ to the critical inverse-temperature. Interestingly,
in the important case of the Sherrington-Kirkpatrick model $\beta_{m}=\beta_{c}$.
In this work we consider the multi-species spherical mixed $p$-spin
models without external field, and characterize by a simple condition
the models for which the second moment method works in the whole replica
symmetric phase, namely, models such that $\beta_{m}=\beta_{c}$. In particular, for those models we obtain the value of $\beta_c$. 
\end{abstract}

\maketitle

\section{Introduction}

The high-temperature  phase of a mixed $p$-spin spin glass model (with
no external field) consists of inverse-temperatures $\beta\leq\beta_{c}$
such that, for large $N$, 
\begin{equation}
\E F_{N,\beta}:=\frac{1}{N}\E\log Z_{N,\beta}\approx\frac{1}{N}\log\E Z_{N,\beta}.\label{eq:apx}
\end{equation}
Computing the mean of the partition function $Z_{N,\beta}$ is trivial, 
and thus for $\beta$ as above one has the mean of the free energy $F_{N,\beta}$
in the large $N$ limit. 
Beyond the expression for the free energy, the system simplifies in this phase in several ways. Most importantly,  
no `replica symmetry
breaking' occurs --- that is, independent samples from the Gibbs measure
are typically roughly orthogonal to each other. 

Despite this, computing the critical value $\beta_{c}$ is generally a  difficult
problem. For general mixed $p$-spin models, the only known method
to achieve its value heavily relies on the Parisi formula \cite{ParisiFormula,Parisi,Talag2},
one of the deepest, most complicated results in mean-field spin glass
theory. See the works of Chen \cite{MR3988771} and Talagrand \cite{Talag}
where a characterization for sub-critical inverse-temperatures $\beta$
is derived from the optimality criterion for the Parisi distribution,
for (single-species) models with Ising and spherical spins respectively. 

In sharp contrast to the usage of the Parisi formula in its extreme
simplicity, an application of the second moment method to the partition
function $Z_{N,\beta}$ very easily allows one to lower bound the
critical inverse-temperature $\beta_{c}$. Interestingly, for the
important SK model \cite{SK75}, the bound actually gives the correct
critical value $\beta_{c}$. In this paper we focus on the question:
when does the second moment work up to the critical $\beta_{c}$?
We answer this question for the multi-species spherical mixed $p$-spin
models. 

\subsection{Definition of the model}

Consider a finite set of species $\S$, which will be fixed throughout
the paper. For each $N\geq1$, we will suppose that
\[
\{1,\ldots,N\}=\bigcup_{s\in\S}I_{s},\quad\text{for some disjoint }I_{s}.
\]
The subsets $I_{s}$, of course, vary with $N$. Denoting $N_{s}:=|I_{s}|$,
we will assume that the proportion of each species converges
\[
\lim_{N\to\infty}\frac{N_{s}}{N}=\lambda(s)\in(0,1),\quad\text{for all }s\in\S.
\]
Let $S(d)=\{\bx\in\R^{d}:\,\|\bx\|=\sqrt{d}\}$ be the sphere of radius
$\sqrt{d}$ in dimension $d$. The \emph{configuration space} of the
multi-species spherical mixed $p$-spin model is
\[
S_{N}=\left\{ (\sigma_{1},\ldots,\sigma_{N})\in\R^{N}:\,\forall s\in\S,\,(\sigma_{i})_{i\in I_{s}}\in S(N_{s})\right\} .
\]

Denoting $\mathbb{Z}_{+}:=\{0,1,\ldots\}$ and $|p|:=\sum_{s\in\S}p(s)$
for $p\in\mathbb{Z}_{+}^{\S}$, let
\[
P=\Big\{ p\in\mathbb{Z}_{+}^{\S}:\,|p|\geq2\Big\}.
\]
Given some nonnegative numbers $(\Delta_{p})_{p\in P}$, define the
\emph{mixture polynomial} in $x=(x(s))_{s\in\S}\in\R^{\S}$,
\begin{equation}
\xi(x)=\sum_{p\in P}\Delta_{p}^{2}\prod_{s\in\S}x(s)^{p(s)}.\label{eq:mixture}
\end{equation}
We will assume that $\xi(1+\epsilon)<\infty$ for some $\epsilon>0$,
where for $a\in\R$ we write $\xi(a)$ for the evaluation of $\xi$
at the constant function $x\equiv a$. 

The multi-species mixed $p$-spin \emph{Hamiltonian} $H_{N}:S_{N}\to\R$
corresponding to the mixture $\xi$ is given by
\begin{equation}
H_{N}(\bs)=\sqrt{N}\sum_{k=2}^{\infty}\sum_{i_{1},\dots,i_{k}=1}^{N}\Delta_{i_{1},\dots,i_{k}}J_{i_{1},\dots,i_{k}}\sigma_{i_{1}}\cdots\sigma_{i_{k}},\label{eq:Hamiltonian-1}
\end{equation}
where $J_{i_{1},\dots,i_{k}}$ are i.i.d. standard normal variables
and if $\#\{j\leq k:\,i_{j}\in I_{s}\}=p(s)$ for any $s\in\S$, then
$\Delta_{i_{1},\dots,i_{k}}=\Delta_{i_{1},\dots,i_{k}}(N)$ is defined
by
\begin{equation}
\Delta_{i_{1},\dots,i_{k}}^{2}=\Delta_{p}^{2}\frac{\prod_{s\in\S}p(s)!}{|p|!}\prod_{s\in\S}N_{s}^{-p(s)}.\label{eq:Delta}
\end{equation}
By a straightforward calculation, the covariance function of $H_{N}(\bs)$
is given by
\begin{equation}
\E H_{N}(\bs)H_{N}(\bs')=N\xi(R(\bs,\bs')),\label{eq:cov}
\end{equation}
where we define the overlap vector
\[
R(\bs,\bs'):=\big(R_{s}(\bs,\bs')\big)_{s\in\S},\quad R_{s}(\bs,\bs'):=N_{s}^{-1}\sum_{i\in I_{s}}\sigma_{i}\sigma_{i}'.
\]

Identifying $S_{N}$ with the product space $\prod_{s\in\S}S(N_{s})$,
let $\mu$ be the product of the uniform measures on each of the spheres
$S(N_{s})$. The \emph{partition function} and \emph{free energy}
at inverse-temperature $\beta\geq0$ are, respectively, defined by
\begin{equation}
Z_{N,\beta}:=\int_{S_{N}}e^{\beta H_{N}(\bs)}d\mu(\bs)\text{\ensuremath{\quad and\quad}}F_{N,\beta}:=\frac{1}{N}\log Z_{N,\beta}.\label{eq:Fbeta}
\end{equation}
The \emph{Gibbs measure} is the random probability measure on $S_{N}$ with density
\[
\frac{dG_{N,\beta}}{d\mu}(\bs)=Z_{N,\beta}^{-1}e^{\beta H_{N}(\bs)}.
\]
If $|\S|=1$, all the definitions above coincide with the usual (single-species)
spherical mixed $p$-spin model.

By Jensen's inequality, 
\begin{equation}
\E F_{N,\beta}=\frac{1}{N}\E\log Z_{N,\beta}\leq\frac{1}{N}\log\E Z_{N,\beta}=\frac{1}{2}\beta^{2}\xi(1).\label{eq:Jensen}
\end{equation}
It is not difficult to check (also using Jensen's inequality, see
Lemma \ref{lem:RS} below) that for any $\beta'<\beta$, 
\[
\E F_{N,\beta}\leq\E F_{N,\beta'}+\frac{1}{2}(\beta^{2}-\beta'^{2})\xi(1).
\]
Hence, there exists a critical inverse-temperature $\beta_{c}$ such
that 
\begin{equation}
\lim_{N\to\infty}\E F_{N,\beta}=\frac{1}{2}\beta^{2}\xi(1)\iff\beta\leq\beta_{c}.\label{eq:bcdef}
\end{equation}

\subsection{The second moment method}

By the Paley\textendash Zygmund inequality, if we are able to show
for some $\beta$ that 
\begin{equation}
\lim_{N\to\infty}\frac{1}{N}\log\E Z_{N,\beta}^{2}=\lim_{N\to\infty}\frac{1}{N}\log\Big((\E Z_{N,\beta})^{2}\Big)=\beta^{2}\xi(1),\label{eq:2ndm}
\end{equation}
then for any $\theta\in(0,1)$, $Z_{N,\beta}\geq\theta\E Z_{N,\beta}$
with probability not exponentially small in $N$. Combined with the
well-known concentration of the free energy, (see e.g., \cite[Theorem 1.2]{PanchenkoBook})
this easily implies that $\beta\leq\beta_{c}$. 

By Fubini's theorem and symmetry we have that
\begin{align*}
\E Z_{N,\beta}^{2} & =\int\int\E\exp\left\{ \beta H_{N}(\bs)+\beta H_{N}(\bs')\right\} d\mu(\bs)d\mu(\bs')\\
 & =\int\exp\left\{ N\beta^{2}(\xi(1)+\xi(R(\bs,\bs_{0}))\right\} d\mu(\bs),
\end{align*}
where $\bs_{0}\in S_{N}$ is an arbitrary point. Using the coarea
formula, one can then check that
\[
\E Z_{N,\beta}^{2}=\int_{[-1,1]^{\S}}\prod_{s\in\S}\frac{\omega_{N_{s}-1}}{\omega_{N_{s}}}\left(1-r(s)^{2}\right)^{\frac{N_{s}-3}{2}}e^{N\beta^{2}(\xi(1)+\xi(r))}dr,
\]
where $\omega_{d}$ denotes the volume of the unit sphere in $\R^{d}$.
Therefore
\begin{equation}
\lim_{N\to\infty}\frac{1}{N}\log\E Z_{N,\beta}^{2}=\beta^{2}\xi(1)+\max_{r\in[0,1)^{\S}}f_{\beta}(r),\label{eq:Z2}
\end{equation}
where 
\[
f_{\beta}(r):=\frac{1}{2}\sum_{s\in\S}\lambda(s)\log(1-r(s)^{2})+\beta^{2}\xi(r).
\]

If we define the threshold inverse-temperature
\[
\beta_{m}:=\max\left\{ \beta\geq0:\,\max_{r\in[0,1)^{\S}}f_{\beta}(r)=f_{\beta}(0)=0\right\} ,
\]
then (\ref{eq:2ndm}) holds if and only if $\beta\leq\beta_{m}$,
from which we have that $\beta_{m}\leq\beta_{c}$. 

The second moment method similarly works for models with Ising spins
up to a threshold $\beta_{m}$ as above, if one appropriately modifies
the logarithmic entropy term in the definition of $f_{\beta}(r)$.
For the SK model, Talagrand used the method in \cite[Section 2]{Talagrand1998}
to show that $\beta_{m}=1/\sqrt{2}$ (the fact that $\beta_{c}\geq1/\sqrt{2}$
was first proved in \cite{AizenmanLebowitzRuelle}). Exploiting special
properties of the SK model, Comets proved in \cite{Comets96} that
$\beta_{c}\leq1/\sqrt{2}$ and therefore $\beta_{c}=\beta_{m}$. For
the $p$-spin generalization of the SK models, Talagrand \cite{Talagrand2000}
used a truncated second moment argument to prove a lower bound for
the critical $\beta_{c}$. Bolthausen \cite{BolthausenMorita} applied
the second moment method conditional on an event related to the TAP
equations to compute the free energy of the SK model with an external
field at high-temperature. 

In the context of the spherical models but in a different direction
than the above, the second moment method was used in the study of
critical points. In \cite{AuffingerGold,geometryMixed,Kivimae,2nd,2ndarbitraryenergy}
it was applied to the complexity of critical points to prove its concentration
around the mean \cite{ABA2,A-BA-C,McKennaComplexity}. 

\subsection{Main results}

Our main result is the following theorem. We will make the following
assumption,
\[
x\in[0,1]^{\S}\setminus\{0\}\implies\xi(x)>0.\tag{\ensuremath{\Asmnb}}
\]
 In the single-species case $|\S|=1$, which is also covered by our
results, the assumption always holds.
\begin{thm}
\label{thm:main}Assuming \Asm, $\beta_{m}=\beta_{c}$ if and only
if 
\begin{equation}
\Big(\frac{d}{dr(s)}\frac{d}{dr(t)}f_{\beta_{m}}(0)\Big)_{s,t\in\S}\text{\ \ is a singular matrix.}\label{eq:HessMat}
\end{equation}
\end{thm}

Since $\big(\frac{d}{dr(s)}f_{\beta}(0)\big)_{s\in\S}=0$, the matrix
in (\ref{eq:HessMat}) determines the local behavior of $f_{\beta_{m}}(r)$
around the origin. Roughly speaking, the theorem says that the second moment
method works up to the critical inverse-temperature if for $\beta$
slightly above $\beta_{m}$ the condition that $\max_{r\in[0,1)^{\S}}f_{\beta}(r)=0$
is broken around $r=0$. 

The proof that $\beta_{m}<\beta_{c}$ if the matrix in (\ref{eq:HessMat})
is regular is the easy part of the theorem. It will follow by showing
that in this case $\beta_{m}<\tilde{\beta}_{m}\leq\beta_{c}$ for
some other threshold $\tilde{\beta}_{m}$, which will arise from applying
the\emph{ }second moment method to a random variable different from
$Z_{N,\beta}$. The argument essentially uses the same idea as in
the truncated second moment method used by Talagrand in \cite{Talagrand2000}. 

To prove the main part of the theorem, concerning the case that the matrix
(\ref{eq:HessMat}) is singular, we will prove the following
proposition.
\begin{prop}
\label{prop:singular}Assuming \Asm, if for some $\beta$,
\begin{equation}
\Big(\frac{d}{dr(s)}\frac{d}{dr(t)}f_{\beta}(0)\Big)_{s,t\in\S}\text{\ \ has a non-negative eigenvalue,}\label{eq:mat}
\end{equation}
then $\beta_{c}\leq\beta$.
\end{prop}

For the spherical single-species mixed $p$-spin models, the Parisi
formula for the limit of the free energy was proved by Talagrand \cite{Talag}
for models with even interactions, and later generalized to arbitrary
mixtures by Chen \cite{Chen}. For the multi-species spherical mixed
$p$-spin models, the Parisi formula was recently proved by Bates
and Sohn \cite{BatesSohn2,BatesSohn1}, assuming that the mixture
polynomial $\xi(x)$ is convex on $[0,1]^{\S}$. For the proof of
the Parisi formula for models with Ising spins, see the works of Talagrand
and Panchenko for the single-species case \cite{Panch,Talag2} and
the work of Panchenko \cite{PanchenkoMulti} for the multi-species
SK model, where it was also assumed that $\xi(x)$ is convex.

For the single-species spherical models, Talagrand also proved in
\cite{Talag} the following characterization of sub-critical inverse-temperatures,
using the Parisi formula: $\beta\leq\beta_{c}$ if and only if
\begin{equation}
\forall r\in[0,1):\quad g_{\beta}(r):=\log(1-r)+r+\beta^{2}\xi(r)\leq0.\label{eq:TalagOpt}
\end{equation}
Using that $\log(1-r)=-\sum_{n\geq1}r^{n}/n$, we have that in the
single-species case 
\begin{equation}
\forall r\in[0,1):\quad f_{\beta}(r)=g_{\beta}(r)+\frac{r^{3}}{3}+\frac{r^{5}}{5}+\frac{r^{7}}{7}+\cdots.\label{eq:g}
\end{equation}
Reassuringly, this is consistent with Theorem \ref{thm:main}. Possibly,
an analogue of the characterization of the high-temperature phase
from \cite{Talag} for multi-species models can be proved using the
results of \cite{BatesSohn2,BatesSohn1}, assuming the convexity of
$\xi(x)$. 

For single-species spherical models, one can deduce Theorem \ref{thm:main}
from the characterization (\ref{eq:TalagOpt}) and (\ref{eq:g}).
Still, even in this setting, it is interesting to understand from
basic principles rather than the Parisi formula when does a basic
tool like the second moment method fails or succeeds. For the multi-species
models, our main result allows one to compute $\beta_{c}$ for a certain
class of models. This result is new, as no analogue for the criterion
from \cite{Talag} is known in this case. In particular, this class
includes models which do not saisfy the assumption that $\xi(x)$
is convex\footnote{For example, if $\Delta_{p}>0$ only when $|p|=2$ then $\xi(x)$ is
a quadratic function, from which one can verify that (\ref{eq:HessMat})
holds. Of course, the coefficients $\Delta_{p}$ with $|p|=2$ can
be chosen so that $\xi(x)$ is not convex. If $f_{\beta}(r)$
has a unique maximum at $r=0$, one can add to such a mixture positive
small coefficients $\Delta_{p}>0$ with $|p|\geq3$ such that 
also after the addition the mixture satisfies (\ref{eq:HessMat}).} as in the proof of the Parisi formula in the multi-species setting
\cite{Barra14,BatesSohn1,PanchenkoMulti}. Finally, one of the main
motivations for this work is that our results are crucial to \cite{TAPmulti2}
where we compute the free energy for pure multi-species spherical
models using the TAP representation developed in \cite{TAPmulti1} (in particular the result we prove in the Appendix).

It is well-known (see Section \ref{sec:PfMainProp}) that for any
$\beta\leq\beta_{c}$ and $\epsilon>0$, with probability going to $1$
as $N\to\infty$,
\[
F_{N,\beta}\approx\frac{1}{N}\log\int_{L_{\beta}(\epsilon)}e^{\beta H_{N}(\bs)}d\mu(\bs),
\]
where we define the subset
\[
L_{\beta}(\epsilon):=\Big\{\bs\in S_{N}:\,\frac{1}{N}H_{N}(\bs)\in(\beta\xi(1)-\epsilon,\beta\xi(1)+\epsilon)\Big\}.
\]
Moreover, with high probability,
\begin{equation}
\frac{1}{N}\log\mu\big(L_{\beta}(\epsilon)\big)\approx\frac{1}{N}\log\E\mu\big(L_{\beta}(\epsilon)\big).\label{eq:Lepsapx}
\end{equation}

For a point $\bs\in S_{N}$, an overlap vector $r\in[-1,1]^{\S}$
and width $\delta>0$, define the subset
\[
B(\bs,r,\delta):=\Big\{\bs'\in S_{N}:\,\forall s\in\S,\,\big|R_{s}(\bs,\bs')-r(s)\big|\leq\delta\Big\}.
\]
We will show using (\ref{eq:Lepsapx}) that for most points $\bs$
in $L_{\beta}(\epsilon)$, the free energy on $B(\bs,r,\delta)$,
namely
\begin{equation}
\frac{1}{N}\log\int_{B(\bs,r,\delta)}e^{\beta H_{N}(\bs')}d\mu(\bs'),\label{eq:Fband-1}
\end{equation}
is close to its conditional expectation given $H_{N}(\bs)$ (see Lemma
\ref{lem:TypCond}). By estimating this conditional expectation, we will prove the following
in Section \ref{sec:PfMainProp}.
\begin{prop}
\label{prop:main} Assume \Asm{} and that $\Delta_p\neq0$ for finitely many $p$. Let $\beta\leq\beta_{c}$, $t>0$. Then for any $r\in[0,1)^{\S}$  and sufficiently small $\delta,\,\epsilon>0$, 
\begin{equation}
\lim_{N\to\infty}\P\bigg(\mu\big(L_{\beta}(\epsilon)\cap A_{r}(t)\big)\geq(1-e^{-Nc})\cdot\mu\big(L_{\beta}(\epsilon)\big)>0
\bigg)=1,\label{eq:LA}
\end{equation}
where $A_{r}(t)\subset S_{N}$ is the set of points $\bs$ such that 
\begin{equation}
\Big|\frac{1}{N}\log\int_{B(\bs, r,\delta)}e^{\beta H_{N}(\bs')}d\mu(\bs')-\frac{1}{2}\beta^{2}\xi(1)-f_{\beta}( r)\Big|<\kappa(r^2 )+t,\label{eq:Fband}
\end{equation}
for some constant $c=c(t,\xi)>0$ and function $\kappa(x)=\kappa(x,\xi,\beta)$ of $x\in[0,1)^\S$ such that $\kappa(0)=0$ and whose directional derivatives at the origin in any direction $x$ are zero 
\begin{equation}
\lim_{\epsilon\to0^+}\frac{\kappa(\epsilon x)}{\epsilon}=0.
\label{eq:kappa0}
\end{equation}
\end{prop}

Obviously, the free energy (\ref{eq:Fband-1}) on $B(\bs, r,\delta)$
lower bounds the total free energy $F_{N,\beta}$. Hence, for $\beta\leq\beta_{c}$,
the proposition in particular gives us a lower bound for the free
energy $F_{N,\beta}$ by using only one point from the set $L_{\beta}(\epsilon)\cap A_{r}(t)$,
with high probability. To prove Proposition \ref{prop:singular},
assuming (\ref{eq:mat}) we will show that if there were some inverse-temperature
$\beta<\beta'<\beta_{c}$, then this lower bound would imply that
$\E F_{N,\beta'}>\frac{1}{2}\beta'^{2}\xi(1)$ in contradiction to
(\ref{eq:Jensen}). 

In Section \ref{sec:mainproofs} we prove Proposition \ref{prop:singular}
and Theorem \ref{thm:main}, assuming Proposition
\ref{prop:main} which is be proved in Section \ref{sec:PfMainProp}
that occupies the rest of the paper.

\section{\label{sec:mainproofs}Proof of the main results assuming Proposition \ref{prop:main}}

In this section we prove Theorem \ref{thm:main} and Proposition \ref{prop:singular},
assuming Proposition \ref{prop:main} which we will prove in Section
\ref{sec:PfMainProp}. 

\subsection{Proof of Proposition \ref{prop:singular}}

We will first prove the proposition assuming that $\Delta_p\neq0$ for only finitely many $p$, in which case we may use Proposition \ref{prop:main}.

Write $f_{\beta}(r)=f_{0}(r)+\beta^{2}\xi(r)$ where 
\[
f_{0}(r)=\frac{1}{2}\sum_{s\in\S}\lambda(s)\log(1-r(s)^{2}).
\]
Note that for any $r$ and real $\alpha$,
\[
\frac{d}{d\alpha}\Big|_{\alpha=0}f_{\beta}(\alpha r)=\frac{d}{d\alpha}\Big|_{\alpha=0}f_{0}(\alpha r)=\frac{d}{d\alpha}\Big|_{\alpha=0}\xi(\alpha r)=0
\]
and
\[
\frac{d^{2}}{d\alpha^{2}}\Big|_{\alpha =0}f_{0}(\alpha r)=-\sum_{s\in\S}\lambda(s)r(s)^{2}<0,
\]
where as usual $(\alpha r)(s)=\alpha r(s)$.  Therefore, for any $\beta<\beta'$,
\begin{equation}
	\frac{d^{2}}{d\alpha^{2}}\Big|_{\alpha=0}f_{\beta}(\alpha r)\geq0 \implies \frac{d^{2}}{d\alpha^{2}}\Big|_{\alpha=0}f_{\beta'}(\alpha r)>0.\label{eq:r}
\end{equation}

Let $\beta$ be some inverse-temperature and assume that the matrix
in (\ref{eq:mat}) has some non-negative eigenvalue. Then there exists
some $r$ with $\|r\|=1$ such that the inequality on the left-hand side of \eqref{eq:r} holds.
From (\ref{eq:mixture}), it is easy to see that we may assume that
this $r$ belongs to $[0,1)^{\S}$. 

Assume towards contradiction that $\beta<\beta_{c}$ and let $\beta'\in(\beta,\beta_{c})$.
Then with the same $r$, 
\[
\frac{d}{d\alpha}\Big|_{\alpha=0}f_{\beta'}(\alpha r)=0\text{\ensuremath{\quad and\quad}}\frac{d^{2}}{d\alpha^{2}}\Big|_{\alpha=0}f_{\beta'}(\alpha r)>0.
\]
Hence, we may choose some small enough $\alpha$ and $t>0$ such that 
\[
f_{\beta'}(\alpha r)>\kappa(\alpha^2r^{2})+2t,
\]
where $\kappa(x)$ is the function from Proposition \ref{prop:main}. 

By the latter proposition, for some small $\delta>0$  with probability
going to $1$ as $N\to\infty$, there exists a point $\bs\in S_{N}$
such that 
\[
F_{N,\beta'}\geq\frac{1}{N}\log\int_{B(\bs,\alpha r,\delta)}e^{\beta'H_{N}(\bs')}d\mu(\bs')>\frac{1}{2}\beta'^{2}\xi(1)+t.
\]

Combined with the well-known concentration of the free energy (see
\cite[Theorem 1.2]{PanchenkoBook}), this contradicts (\ref{eq:Jensen}).
We therefore conclude that $\beta_{c}\leq\beta$.

It remains to prove the proposition in the case where infinitely many $\Delta_p$ are non-zero.
Given some $\xi(x)$ consider the mixture $\bar \xi(X)$ obtained by replacing  by zero all the coefficients $\Delta_p$ whenever $|p|\geq3$. Note that the matrix in \eqref{eq:mat} is determined by the coefficients $\Delta_p$ with $|p|=2$ only. Moreover, by the following lemma, the critical inverse-temperature of $\xi(x)$ is less than or equal to that of $\bar\xi(x)$. From this, the proposition follows also when infinitely many $\Delta_p$ are non-zero.

\begin{lem}
	\label{lem:RS}Suppose that 
	\[
	\xi(x)=\sum_{p:|p|\geq1}\Delta_{p}^{2}\prod_{s\in\S}x(s)^{p(s)},\quad\bar{\xi}(x)=\sum_{p:|p|\geq1}\bar{\Delta}_{p}^{2}\prod_{s\in\S}x(s)^{p(s)}
	\]
	are two mixtures such that $\Delta_{p}^{2}\geq\bar{\Delta}_{p}^{2}$
	for any $p$. Let $F_{N,\beta}$ and $\bar{F}_{N,\beta}$ their corresponding
	free energies. Then 
	\[
	\lim_{N\to\infty}\E F_{N,\beta}=\frac{1}{2}\beta^{2}\xi(1)\implies\lim_{N\to\infty}\E\bar{F}_{N,\beta}=\frac{1}{2}\beta^{2}\bar{\xi}(1).
	\]
\end{lem}

\begin{proof}
	Define the mixture 
	\[
	\hat{\xi}(x)=\xi(x)-\bar{\xi}(x)=\sum_{p:|p|\geq1}(\Delta_{p}^{2}-\bar{\Delta}_{p}^{2})\prod_{s\in\S}x(s)^{p(s)}.
	\]
	Denote by $H_{N}(\bs)$, $\bar{H}_{N}(\bs)$ and $\hat{H}_{N}(\bs)$
	the Hamiltonians corresponding to $\xi(x)$, $\bar{\xi}(x)$ and $\hat{\xi}(x)$,
	respectively. Define $\bar{H}_{N}(\bs)$ and $\hat{H}_{N}(\bs)$ on
	the same probability space such that they are independent. Note that
	in distribution (as processes) 
	\[
	H_{N}(\bs)=\bar{H}_{N}(\bs)+\hat{H}_{N}(\bs),
	\]
	since the covariance functions and expectation of the Gaussian processes
	in both sides are equal. 
	
	By Jensen's inequality, 
	\begin{align*}
		\E\log\int_{S_{N}}e^{\beta H_{N}(\bs)}d\mu(\bs) & =\E\E\Big(\log\int_{S_{N}}e^{\beta(\bar{H}_{N}(\bs)+\hat{H}_{N}(\bs))}d\mu(\bs)\,\Big|\,(\bar{H}_{N}(\bs))_{\bs\in S_{N}}\Big)\\
		& \leq\E\log\int_{S_{N}}e^{\beta\bar{H}_{N}(\bs)}d\mu(\bs)+\frac{1}{2}N\beta^{2}\hat{\xi}(1).
	\end{align*}
	Hence, if $\lim_{N\to\infty}\E F_{N,\beta}=\frac{1}{2}\beta^{2}\xi(1)$,
	then
	\[
	\liminf_{N\to\infty}\E\bar{F}_{N,\beta}\geq\frac{1}{2}\beta^{2}\xi(1)-\frac{1}{2}\beta^{2}\hat{\xi}(1)=\frac{1}{2}\beta^{2}\bar{\xi}(1).
	\]
	The matching upper bound follows from (\ref{eq:Jensen}).
\end{proof}
\subsection{Proof of Theorem \ref{thm:main}}

Recall that $\beta_{m}\leq\beta_{c}$. If the matrix 
\begin{equation}
\Big(\frac{d}{dr(s)}\frac{d}{dr(t)}f_{\beta_{m}}(0)\Big)_{s,t\in\S}\label{eq:Hessfb}
\end{equation}
is singular, then by Proposition \ref{prop:singular}, $\beta_{c}\leq\beta_{m}$
and thus $\beta_{c}=\beta_{m}$.

Henceforth, assume that the matrix above is regular. We will prove
that in this case $\beta_{m}<\beta_{c}$. As mentioned in the introduction,
the argument we  use is essentially equivalent to the second moment
method with truncation used in \cite{Talagrand2000}. 

Fix some arbitrary $\gamma>0$. For $\beta>0$, consider the random
variable
\[
U_{N,\beta}:=\mu\Big(\Big\{\bs\in S_{N}:\,\Big|\frac{1}{N}H_{N}(\beta)-\beta\xi(1)\Big|<N^{-\gamma}\Big\}\Big).
\]
It is easy to check that
\[
\lim_{N\to\infty}\frac{1}{N}\log\E U_{N,\beta}=-\frac{1}{2}\beta^{2}\xi(1),
\]
and similarly to (\ref{eq:Z2}),
\[
\lim_{N\to\infty}\frac{1}{N}\log\E U_{N,\beta}^{2}=-\beta^{2}\xi(1)+\max_{r\in[0,1)^{\S}}\tilde{f}_{\beta}(r),
\]
where
\[
\tilde{f}_{\beta}(r):=\frac{1}{2}\sum_{s\in\S}\lambda(s)\log(1-r(s)^{2})+\beta^{2}\frac{\xi(1)\xi(r)}{\xi(1)+\xi(r)}.
\]

Define the inverse-temperature
\[
\tilde{\beta}_{m}:=\max\left\{ \beta\geq0:\,\max_{r\in[0,1)^{\S}}\tilde{f}_{\beta}(r)=\tilde{f}_{\beta}(0)=0\right\} .
\]
For any $\beta\leq\tilde{\beta}_{m}$, 
\[
\lim_{N\to\infty}\frac{1}{N}\log\E U_{N,\beta}^{2}=\lim_{N\to\infty}\frac{1}{N}\log\Big((\E U_{N,\beta})^{2}\Big)=-\beta^{2}\xi(1).
\]
From the Paley\textendash Zygmund inequality, for such $\beta$ and
any $\theta\in(0,1)$, with probability not exponentially small in
$N$, $U_{N,\beta}\geq\theta\E U_{N,\beta}$. On this event, $F_{N,\beta}\geq\frac{1}{2}\beta^{2}\xi(1)+o(1)$.
Hence, from the well-known concentration of the free energy and (\ref{eq:Jensen}),
$\tilde{\beta}_{m}\leq\beta_{c}$. 

Note that $\beta_{m}\leq\tilde{\beta}_{m}$, since
\begin{equation}
\begin{aligned}\tilde{f}_{\beta}(r)<f_{\beta}(r) & \iff\xi(r)\neq0,\\
\tilde{f}_{\beta}(r)=f_{\beta}(r) & \iff\xi(r)=0.
\end{aligned}
\label{eq:ftildebd}
\end{equation}

Recall that we assume that the Hessian matrix (\ref{eq:Hessfb}) is
regular.  By the definition of $\beta_{m}$, since 
\begin{equation}
\forall s\in\S,\quad\frac{d}{dr(s)}f_{\beta}(0)=0,\label{eq:grad0}
\end{equation}
all the eigenvalues of the Hessian (\ref{eq:Hessfb}) of $f_{\beta_{m}}(r)$ at $r=0$  are strictly negative. Choose 
some $\alpha>0$ such that they are all less than
$-\alpha$. From continuity, the eigenvalues of the Hessian of $f_{\beta}(r)$
at $r=0$ are less than $-\alpha/2$ for all $\beta$ in some right
neighborhood of $\beta_{m}$. Combined with (\ref{eq:grad0}), for
such $\beta$, this implies that 
\[
\max_{r\in A\cap[0,1)^{\S}}\tilde{f}_{\beta}(r)\leq\max_{r\in A\cap[0,1)^{\S}}f_{\beta}(r)=0,
\]
for some open neighborhood $A$ of $r=0$. 

For such $\beta$ and
some small enough $\delta>0$,
\[
\max_{r\in[0,1)^{\S}\setminus[0,1-\delta]^{\S}}\tilde{f}_{\beta}(r)<0.
\]

Recall the assumption \Asm. Since $[0,1-\delta]^{\S}\setminus A$
is closed, from (\ref{eq:ftildebd}) and the continuity of $f_{\beta}(r)$
and $\tilde{f}_{\beta}(r)$ in $r$ and $\beta$, for some small $c>0$
and $\beta$ close enough to $\beta_{m}$,
\[
\max_{r\in[0,1-\delta]^{\S}\setminus A}\tilde{f}_{\beta}(r)<\max_{r\in[0,1-\delta]^{\S}\setminus A}f_{\beta}(r)-2c
\]
and 
\[
\max_{r\in[0,1-\delta]^{\S}\setminus A}f_{\beta}(r)<\max_{r\in[0,1-\delta]^{\S}\setminus A}f_{\beta_{m}}(r)+c\leq c.
\]
Hence,
\[
\max_{r\in[0,1-\delta]^{\S}\setminus A}\tilde{f}_{\beta}(r)<-c.
\]

Combining the above, we have that $\beta\leq\tilde{\beta}_{m}$ for
any $\beta$ in some small right neighborhood of $\beta_{m}$. Therefore,
$\beta_{m}<\tilde{\beta}_{m}\leq\beta_{c}$, which completes the proof.\qed

\section{\label{sec:PfMainProp}Proof of Proposition \ref{prop:main}}

To prove Proposition \ref{prop:main}, we will need the three auxiliary
results below. The first is an elementary well known result about the
volume of approximate level sets, or entropy for sub-critical $\beta$.
\begin{lem}
\label{lem:vol}If $\beta\leq\beta_{c}$, for any $t>0$, for small enough $\epsilon>0$,
there exists some $c>0$ such that for large $N$,
\[
\P\left(\Big|\frac{1}{N}\log\mu(L_{\beta}(\epsilon))+\frac{1}{2}\beta^{2}\xi(1)\Big|<t\right)>1-e^{-Nc}.
\]
\end{lem}

Define the random fields
\[
\phi_{N,\beta}(\bs,r,\delta)=\frac{1}{N}\log\int_{B(\bs,r,\delta)}e^{\beta H_{N}(\bs')}d\mu(\bs')
\]
and
\begin{equation}
\varphi_{N,\beta}(\bs,r,\delta)=\E\Big(\phi_{N,\beta}(\bs,r,\delta)\,\Big|\,H_{N}(\bs)\Big),\label{eq:varphi}
\end{equation}
and the random set
\[
D_{N}(r,\delta,t)=\left\{ \bs\in\SN:\,\big|\phi_{N,\beta}(\bs,r,\delta)-\varphi_{N,\beta}(\bs,r,\delta)\big|>t\right\} .
\]

Proposition \ref{prop:main} concerns the volume of points in $L_{\beta}(\epsilon)$
such that $\phi_{N,\beta}(\bs,r,\delta)$ is close to a certain value.
The next lemma shows that $\phi_{N,\beta}(\bs,r,\delta)$ and $\varphi_{N,\beta}(\bs,r,\delta)$
are close to each other on $L_{\beta}(\epsilon)$, up to a subset of small volume. It will allow
us to work with the conditional expectation $\varphi_{N,\beta}(\bs,r,\delta)$,
which only depends on the value of the Hamiltonian at $\bs$.
\begin{lem}
\label{lem:TypCond}Suppose that $\beta\leq\beta_{c}$. For any $r\in(-1,1)^{\S}$
and positive $\delta$  and $t$, for large $N$, 
\[
\P\left(\frac{1}{N}\log\mu\big(L_{\beta}(\epsilon)\cap D_{N}(r,\delta,t)\big)\geq-\frac{1}{2}\beta^{2}\xi(1)+\beta\epsilon-\frac{\epsilon^2}{2\xi(1)}-\frac{t^{2}}{8\xi(1)}\right)\leq e^{-\frac{Nt^{2}}{10\xi(1)}}.
\]
\end{lem}

The main ingredient in the proof of Proposition \ref{prop:main} is
the following estimate on the conditional expectation $\varphi_{N,\beta}(\bs,r,\delta)$.
\begin{prop}
\label{prop:E}Assume \Asm{} and that $\Delta_p\neq0$ for finitely many values of $p$ and $\beta\leq\beta_{c}$.  Then, for
any $r\in[0,1)^{\S}$, almost surely,
\begin{equation}
\begin{aligned}
	\limsup_{N\to\infty}\sup_{\bs\in S_{N}} & \bigg|\frac{\beta}{N}\frac{\xi(r)}{\xi(1)}H_{N}(\bs)-\beta^{2}\xi(r)+f_{\beta}(r)+\frac12\beta^{2}\xi(1)-\varphi_{N,\beta}(\bs,r,\delta)\bigg|\\
 & \leq\beta\delta c_{\xi}+\kappa(r^{2}),
\end{aligned}
\label{eq:PropE}
\end{equation}
for some constant $c_{\xi}$  depending only on $\xi$ and function $\kappa(x)=\kappa(x,\xi,\beta)$ as in Proposition \ref{prop:main}. 
\end{prop}

Next we will prove Proposition \ref{prop:main}, assuming the three
results above. They will be proved in the following subsections. Let
$t,\epsilon,\delta>0$, $r\in[0,1)^{\S}$ and $\alpha\in(0,1)$. Set $c=\frac{1}{3}\frac{(t/2)^{2}}{8\xi(1)}$.
By Lemmas \ref{lem:vol} and \ref{lem:TypCond}, if $\epsilon>0$ is small enough, for some $a>0$ and
large enough $N$, with probability at least $1-e^{-Na}$,
\begin{align*}
\frac{1}{N}\log\mu(L_{\beta}(\epsilon)) & >-\frac{1}{2}\beta^{2}\xi(1)-c,\\
\frac{1}{N}\log\mu\big(L_{\beta}(\epsilon)\cap D_{N}( r,\delta,t/2)\big) & <-\frac{1}{2}\beta^{2}\xi(1)-3c.
\end{align*}
On this event,
\[
\frac{\mu\big(L_{\beta}(\epsilon)\setminus D_{N}( r,\delta,t/2)\big)}{\mu\big(L_{\beta}(\epsilon)\big)}>1-e^{-Nc}.
\]

On $L_{\beta}(\epsilon)$, 
\[
\bigg|\frac{\beta}{N}\frac{\xi(r)}{\xi(1)}H_{N}(\bs)-\beta^{2}\xi(r)\bigg|\leq\beta\epsilon.
\]
By Proposition \ref{prop:E}, a.s., 
\begin{align*}
\limsup_{N\to\infty}\sup_{\bs\in S_{N}\setminus D_{N}(r,\delta,t/2)} & \bigg|\frac{\beta}{N}\frac{\xi(r)}{\xi(1)}H_{N}(\bs)-\beta^{2}\xi(r)+f_{\beta}(r)+\frac12\beta^{2}\xi(1)\\
 & -\phi_{N,\beta}(\bs,r,\delta)\bigg|\leq\beta\delta c_{\xi}+\kappa(r^{2})+\frac{1}{2}t,
\end{align*}
for $c_{\xi}$ and $\kappa(x)$ as in the proposition.

Hence, a.s.,
\begin{align*}
\limsup_{N\to\infty}\sup_{\bs\in L_{\beta}(\epsilon)\setminus D_{N}(r,\delta,t/2)} & \Big|f_{\beta}(r)+\frac{1}{2}\beta^{2}\xi(1)-\phi_{N,\beta}(\bs,r,\delta)\Big|\\
 & <\beta\delta c_{\xi}+\beta^{2}\kappa(r^{2})+\frac{1}{2}t+\beta\epsilon,
\end{align*}
and thus, for small enough $\epsilon$ and $\delta$, 
\[
\lim_{N\to\infty}\P\Big(L_{\beta}(\epsilon)\setminus D_{N}(r,\delta,t/2)\subset A_{r}(t)\Big)=1.
\]
This completes the proof of Proposition \ref{prop:main}. It remains
to prove the three results above.

\subsection{Proof of Lemma \ref{lem:vol}}

Suppose that $\beta\leq\beta_{c}$ and let $\epsilon,t>0$ be arbitrary
numbers. Using Fubini's theorem, one sees that for small enough $\epsilon>0$,
\begin{equation}
\frac{1}{N}\log\E\mu(L_{\beta}(\epsilon))=-\frac{1}{2}\beta^{2}\xi(1)+\frac{t}{2}+o(1).\label{eq:Evol}
\end{equation}
By Markov's inequality,
\[
\P\left(\frac{1}{N}\log\mu(L_{\beta}(\epsilon))>-\frac{1}{2}\beta^{2}\xi(1)+t\right)<e^{-\frac{Nt}{2}+o(N)}.
\]

It remains to show that for some $c>0$ and large $N$, 
\[
\P\left(\frac{1}{N}\log\mu(L_{\beta}(\epsilon))<-\frac{1}{2}\beta^{2}\xi(1)-t\right)<e^{-Nc}.
\]
For $\beta<\beta_c$ close enough to $\beta_c$,
\[
\P\left(\frac{1}{N}\log\mu(L_{\beta^c}(\epsilon))<-\frac{1}{2}\beta_c^{2}\xi(1)-t\right)<\P\left(\frac{1}{N}\log\mu(L_{\beta}(\epsilon/2))<-\frac{1}{2}\beta^{2}\xi(1)-t\right),
\]
hence it will be enough to prove the inequality for $\beta<\beta_c$.

Assume towards contradiction that for any $c>0$, for some large as we wish $N$, 
\begin{equation}
	\label{eq:bd2203}
\P\left(\frac{1}{N}\log\mu(L_{\beta}(\epsilon))<-\frac{1}{2}\beta^{2}\xi(1)-t\right)>e^{-Nc}.
\end{equation}
Let $\beta_0$ be some inverse-temperature. On the event in \eqref{eq:bd2203},
\begin{equation}
	\label{eq:bd2203-02}
\begin{aligned}
	\frac{1}{N}\log\int_{  L_{\beta}(\epsilon)}e^{\beta_{0}H_{N}(\bs)}d\mu(\bs)  &\leq -\frac12 \beta^2\xi(1)-t+\beta_0(\beta\xi(1)+\epsilon)\\
	&<\frac12\beta_0^2\xi(1)-\frac{t}{2},
\end{aligned}
\end{equation}
where the second inequality holds if we assume that $\beta_0$ is close enough to $\beta$ and $\epsilon$ is sufficiently small.

Assume in addition  that $\beta_0\in(\beta-\epsilon/\xi(1),\beta+\epsilon/\xi(1))$. Then,
\begin{align*}
	& \lim_{N\to\infty}\frac{1}{N}\log\E\int_{S_{N}\setminus  L_{\beta}(\epsilon)}e^{\beta_{0}H_{N}(\bs)}d\mu(\bs) =\max_{y:|y-\beta\xi(1)|\geq \epsilon}\left\{ -\frac{y^{2}}{2\xi(1)}+\beta_{0}y\right\} < \frac{1}{2}\beta_{0}^{2}\xi(1).
\end{align*}
By Markov's inequality, for some $a,b>0$ which depend on $\beta$, $\beta_0$ , $\epsilon$ and $\xi(1)$, for large $N$,
\begin{equation}
\label{eq:bd2203-03}
\P\Big(\frac{1}{N}\log\int_{S_{N}\setminus  L_{\beta}(\epsilon)}e^{\beta_{0}H_{N}(\bs)}d\mu(\bs) > \frac{1}{2}\beta_{0}^{2}\xi(1)-a\Big)<e^{-Nb}.
\end{equation}

Since we assumed that $\beta<\beta_c$, we may also assume that $\beta_0<\beta_c$, and thus $\lim_{N\to\infty}\E F_{N,\beta_0} = \frac12\beta_0^2\xi(1)$.

Combining the above, we have that for some $N$ as large as we wish, both \eqref{eq:bd2203-02} and the complement of the event in \eqref{eq:bd2203-03} occur simultaneously with probability at least $e^{-Nc}-e^{-Nb}$.
For any small $c>0$, for such $N$ we therefore have that
\begin{equation}
\label{eq:2203-04}
\P\left(
F_{N,\beta_0}<\E F_{N,\beta_0}-\min\{t/4,a/2\}
\right)>\frac12 e^{-Nc},
\end{equation}
in contradiction to the well-known concentration of the free energy (see e.g. \cite[Theorem 1.2]{PanchenkoBook}). This completes the proof.
\qed

\subsection{Proof of Lemma \ref{lem:TypCond}}

Let $\bs\in S_{N}$ be an arbitrary point. Conditional on $H_{N}(\bs)$
the value at $\bs$, the Hamiltonian is a Gaussian field whose variance
is bounded by the variance before conditioning $N\xi(1)$. Hence,
from the well-known concentration of the free energy, for any $x\in\R$,
\[
\P\Big(\bs\in D_{N}(r,\delta,t))\,\Big|\,H_{N}(\bs)=x\Big)\leq2e^{-\frac{Nt^{2}}{4\xi(1)}}.
\]

Therefore, from Fubini's Theorem, as $N\to\infty$,
\begin{align*}
 & \E\mu\Big(D_{N}(r,\delta,t)\cap L_{\beta}(\epsilon)\Big)=\P\Big(\bs\in D_{N}(r,\delta,t)\cap L_{\beta}(\epsilon)\Big)\\
 & \leq2e^{-\frac{Nt^{2}}{4\xi(1)}}\P\Big(\bs\in L_{\beta}(\epsilon)\Big)=\exp\left(-N\left(\frac{t^{2}}{4\xi(1)}+\frac{1}{2}\beta^{2}\xi(1)-\beta\epsilon+\frac{\epsilon^2}{2\xi(1)}\right)+o(N)\right).
\end{align*}
By Markov's inequality, 
\[
\pushQED{\qed}\P\left(\mu\big(D_{N}(r,\delta,t)\cap L_{\beta}(\epsilon)\big)\geq e^{-N\left(\frac{t^{2}}{8\xi(1)}+\frac{1}{2}\beta^{2}\xi(1)-\beta\epsilon+\frac{\epsilon^2}{2\xi(1)}\right)}\right)\leq e^{-\frac{Nt^{2}}{8\xi(1)}+o(N)}.\qedhere\popQED
\]

\subsection{Proof of Proposition \ref{prop:E}}

The proof will be based on the three lemmas below which will be proved
in the next subsections. Define
\[
B(\bs,r):=B(\bs,r,0)=\Big\{\bs'\in S_{N}:\,\forall s\in\S,\,R_{s}(\bs,\bs')=r(s)\Big\}.
\]
Note that we may identify $B(\bs,r)$ with the product spheres, one
for each $s\in\S$, of codimension 1 in $S(N_{s})$. Endow each of
those spheres with the uniform probability measure and let $\nu=\nu_{\bs,r}$
denote the product measure on $B(\bs,r)$. Similarly to $\varphi_{N,\beta}(\bs,r,\delta)$,
(see (\ref{eq:varphi})) define 
\[
\varphi_{N,\beta}(\bs,r)=\E\Big(\frac{1}{N}\log\int_{B(\bs,r)}e^{\beta H_{N}(\bs')}d\nu(\bs')\,\Big|\,H_{N}(\bs)\Big).
\]

\begin{lem}
\label{lem:band_to_codim}Let $\beta\geq0$ and $r\in[0,1)^{\S}$.
For large enough $N$, almost surely,
\begin{equation}
\sup_{\bs\in S_{N}}\Big|\varphi_{N,\beta}(\bs,r,\delta)-\varphi_{N,\beta}(\bs,r)-\frac{1}{2}\sum_{s\in\S}\lambda(s)\log(1-r(s)^{2})\Big|\leq\delta c_{\xi}\beta\Big(\max_{\bs\in S_{N}}\frac{|H_{N}(\bs)|}{N}+1\Big),\label{eq:phihat_codimapx}
\end{equation}
where $c_{\xi}>0$ is a constant that only depends on $\xi$.
\end{lem}

For $r\in[0,1)^{\S}$, define 
\begin{equation}
\tilde{\xi}_{r}(x)=\xi((1-r^{2})x+r^{2})-\xi(r^{2}),\label{eq:xitilde_r}
\end{equation}
Here, all operations between functions $\S\to\R$ are performed elementwise,
for example, $((1-r^{2})x)(s):=(1-r(s)^{2})x(s)$. Explicitly, 
\begin{equation}
\begin{aligned}\tilde{\xi}_{r}(x) & =\sum_{p:\,|p|\geq2}\Delta_{p}^{2}\Big(\prod_{s\in\S}\left((1-r(s)^{2})x(s)+r(s)^{2}\right)^{p(s)}-\prod_{s\in\S}r(s)^{2p(s)}\Big)\\
 & =\sum_{p:\,|p|\geq1}\Delta_{p,r}^{2}\prod_{s\in\S}x(s)^{p(s)},
\end{aligned}
\label{eq:xitilde_r_2}
\end{equation}
where
\[
\Delta_{p,r}^{2}:=\sum_{p'\geq p}\Delta_{p'}^{2}\prod_{s\in\S}\binom{p'(s)}{p(s)}(1-r(s)^{2})^{p(s)}r(s)^{2(p'(s)-p(s))},
\]
where we write $p'\geq p$ if $p'(s)\geq p(s)$ for all $s\in\S$. 
\begin{rem}
\label{rem:p1}In the Introduction we defined the multi-species mixtures
(\ref{eq:mixture}) with coefficients for $p$ with $|p|\geq2$ and
their corresponding Hamiltonians in (\ref{eq:Hamiltonian-1}). Of
course, one may consider mixtures with non-zero coefficients also
for $p$ with $|p|=1$, for which the summation in the definition
of the corresponding Hamiltonian in (\ref{eq:Hamiltonian-1}) starts
from $k=1$. 
\end{rem}

Note that $\tilde{\xi}_{r}(x)$ is a mixture as in the remark above
and let $\tilde{H}_{N}^{r}(\bs)$ be the corresponding Hamiltonian.
We remark that the same mixture has been considered in several previous
works in the study of the Gibbs measure \cite{geometryMixed,geometryGibbs}
and in the context of the TAP approach \cite{TAPChenPanchenkoSubag,TAPIIChenPanchenkoSubag,FElandscape,TAPmulti1}.
\begin{lem}
\label{lem:Htilde_r}Let $\beta\geq0$ and $r\in[0,1)^{\S}$. Then,
almost surely, 
\begin{equation}
\lim_{N\to\infty}\sup_{\bs\in S_{N}}\Big|\varphi_{N,\beta}(\bs,r)-\frac{\beta}{N}\frac{\xi(r)}{\xi(1)}H_{N}(\bs)-\frac{1}{N}\E\log\int_{S_{N}}e^{\beta\tilde{H}_{N}^{r}(\bs')}d\mu(\bs')\Big|=0.\label{eq:phihat_codimapx-3}
\end{equation}
\end{lem}

The last lemma we need approximates the free energy of $\tilde{H}_{N}^{r}(\bs)$, for small $r$. 
\begin{lem}
\label{lem:codim_mean}Assume \Asm{} and that $\Delta_p\neq0$ for finitely many values of $p$ and  $\beta\leq \beta_{c}$. Then for any $r\in[0,1)^{\S}$, 
\begin{equation}
\limsup_{N\to\infty}\Big|\frac{1}{N}\E\log\int_{S_{N}}e^{\beta\tilde{H}_{N}^{r}(\bs')}d\mu(\bs')-\frac{1}{2}\beta^{2}{\xi}(1)\Big|\leq\kappa(r^{2}),\label{eq:phihat_codimapx-1}
\end{equation}
where $\kappa(x)=\kappa(x,\xi,\beta)$ is a function as in Proposition \ref{prop:main}.
\end{lem}

Suppose that $\xi(x)$ satisfies the assumption in \Asm{} and let
$\beta\leq \beta_{c}$. Combining the three lemmas above, we have that
for $r\in[0,1)^{\S}$, almost surely,
\begin{align*}
\limsup_{N\to\infty}\sup_{\bs\in S_{N}} & \Big|\varphi_{N,\beta}(\bs,r,\delta)-\Big(\frac{\beta}{N}\frac{\xi(r)}{\xi(1)}H_{N}(\bs)+\frac{1}{2}\sum_{s\in\S}\lambda(s)\log(1-r(s)^{2})+\frac{1}{2}\beta^{2}{\xi}(1)\Big)\Big|\\
 & \leq\delta c_{\xi}\beta\Big(\max_{\bs\in S_{N}}\frac{|H_{N}(\bs)|}{N}+1\Big)+\beta^{2}\kappa(r^{2}),
\end{align*}
where $c_{\xi}$ and $\kappa(x)$ are as in the lemmas. By \cite[Lemma 25]{TAPmulti1},
for some constant $C_{\xi}>0$ that depends only on $\xi$,
\begin{equation}
\E\max_{\bs\in S_{N}}\frac{|H_{N}(\bs)|}{N}\leq C_{\xi}.\label{eq:max}
\end{equation}
By the Borell-TIS inequality and the Borel-Cantelli lemma, almost
surely,
\[
\limsup_{N\to\infty}\max_{\bs\in S_{N}}\frac{|H_{N}(\bs)|}{N}\leq2C_{\xi}.
\]

Proposition \ref{prop:E}
follows by combining the above. It remains to prove the three lemmas
above. This will be done in Subsections \ref{subsec:pf1}-\ref{subsec:pf3}
below.\qed

\subsection{\label{subsec:pf1}Proof of Lemma \ref{lem:band_to_codim}}

Fix some $\bs_{\star}\in S_{N}$. Since the Hamiltonian $H_{N}(\bs)$
is a Gaussian process, it can be decomposed as
\begin{equation}
H_{N}(\bs)=\hat{H}_{N}(\bs)+\eta(\bs)H_{N}(\bs_{\star}),\label{eq:decomp}
\end{equation}
where $\eta(\bs)H_{N}(\bs_{\star})=\E(H_{N}(\bs)\,|\,H_{N}(\bs_{\star}))$
with
\[
\eta(\bs):=\frac{\E\big(H_{N}(\bs)H_{N}(\bs_{\star})\big)}{\E\big(H_{N}(\bs_{\star})^{2}\big)}=\frac{\xi(R(\bs,\bs_{\star}))}{\xi(1)}
\]
and $\hat{H}_{N}(\bs)$ is a centered Gaussian process, independent
of $H_{N}(\bs_{\star})$, with covariance function
\begin{align*}
\frac{1}{N}\E\big(\hat{H}_{N}(\bs)\hat{H}_{N}(\bs')\big) & =\frac{1}{N}\E\big(H_{N}(\bs)H_{N}(\bs')\big)-\frac{1}{N}\frac{\E\big(H_{N}(\bs)H_{N}(\bs_{\star})\big)\E\big(H_{N}(\bs')H_{N}(\bs_{\star})\big)}{\E\big(H_{N}(\bs_{\star})^{2}\big)}\\
 & =\xi(R(\bs,\bs'))-\frac{\xi(R(\bs,\bs_{\star}))\xi(R(\bs',\bs_{\star}))}{\xi(1)}.
\end{align*}

Note that, since $\eta(\bs)=\frac{\xi(r)}{\xi(1)}$ on $B(\bs_{\star},r)$,
\begin{equation}
\varphi_{N,\beta}(\bs_{\star},r)=\E\Big(\frac{1}{N}\log\int_{B(\bs_{\star},r)}e^{\beta\hat{H}_{N}(\bs)}d\nu(\bs)\Big)+\frac{\beta}{N}\frac{\xi(r)}{\xi(1)}H_{N}(\bs_{\star}).\label{eq:Hhat}
\end{equation}

For any $\bs\in B(\bs_{\star},r,\delta)$,
\[
|\eta(\bs)-\eta(\bs_{\star})|\leq\max_{t:\,t(s)\in[r(s)-\delta,r(s)+\delta]}\left|\frac{\xi(t)}{\xi(1)}-\frac{\xi(r)}{\xi(1)}\right|\leq\frac{\delta}{\xi(1)}\sum_{s\in\S}\frac{d}{dx(s)}\xi(1)=:c_{\xi}\delta,
\]
and therefore
\[
\Big|\varphi_{N,\beta}(\bs_{\star},r,\delta)-\E\Big(\frac{1}{N}\log\int_{B(\bs_{\star},r,\delta)}e^{\beta\hat{H}_{N}(\bs)}d\mu(\bs)\Big)-\frac{\beta}{N}\frac{\xi(r)}{\xi(1)}H_{N}(\bs_{\star})\Big|\leq\frac{\beta}{N}c_{\xi}\delta|H_{N}(\bs_{\star})|.
\]

Hence, to prove the lemma it will be enough to show that for some
$c>0$ depending only on $\xi$, for large $N$,
\[
\begin{aligned} & \bigg|\E\frac{1}{N}\log\int_{B(\bs_{\star},r)}e^{\beta\hat{H}_{N}(\bs)}d\nu(\bs)+\frac{1}{2}\sum_{s\in\S}\lambda(s)\log(1-r(s)^{2})\\
 & -\E\frac{1}{N}\log\int_{B(\bs_{\star},r,\delta)}e^{\beta\hat{H}_{N}(\bs)}d\mu(\bs)\bigg|<\delta c\beta.
\end{aligned}
\]

Since the variance of $\hat{H}_{N}(\bs)$ is bounded uniformly in
$\bs$ by $N\xi(1)$, the variance of the unconditional Hamiltonian,
from the concentration of the free energies around their mean (see
\cite[Theorem 1.2]{PanchenkoBook}) it will be enough to show that
\begin{equation}
\begin{aligned} & \bigg|\frac{1}{N}\log\int_{B(\bs_{\star},r)}e^{\beta\hat{H}_{N}(\bs)}d\nu(\bs)+\frac{1}{2}\sum_{s\in\S}\lambda(s)\log(1-r(s)^{2})\\
 & -\frac{1}{N}\log\int_{B(\bs_{\star},r,\delta)}e^{\beta\hat{H}_{N}(\bs)}d\mu(\bs)\bigg|<\delta c\beta
\end{aligned}
\label{eq:diff1-1}
\end{equation}
with probability that goes to $1$ as $N\to\infty$, for $c$ as above.

By \cite[Lemma 25]{TAPmulti1}, for any $C>0$, for some $L>0$ that
depends only on $\xi$, with probability at least $1-e^{-NC}$, for
all $\bs,\,\bs'\in S_{N}$,
\[
\frac{1}{N}\big|H_{N}(\bs)-H_{N}(\bs')\big|\leq L\max_{s\in\S}\sqrt{R_{s}(\bs-\bs',\bs-\bs')}.
\]
For any $\bs,\,\bs'\in S_{N}$,
\[
|\eta(\bs)-\eta(\bs')|\leq\frac{1}{\xi(1)}\sum_{s\in\S}\frac{d}{dx(s)}\xi(1)\cdot\sqrt{R_{s}(\bs-\bs',\bs-\bs').}
\]
Hence, from (\ref{eq:decomp}) and the fact that $H_{N}(\bs_{\star})$
is a Gaussian variable with zero mean and variance $N\xi(1)$, with
the probability going to $1$ as $N\to\infty$, for all $\bs,\,\bs'\in S_{N}$,
\[
\frac{1}{N}\big|\hat{H}_{N}(\bs)-\hat{H}_{N}(\bs')\big|\leq2L\max_{s\in\S}\sqrt{R_{s}(\bs-\bs',\bs-\bs')}.
\]

On this event, (\ref{eq:diff1-1}) holds since by the co-area formula,
\begin{align*}
\int_{B(\bs_{\star},r,\delta)}e^{\beta\hat{H}_{N}(\bs)}d\mu(\bs) & =\int_{B(\bs_{\star},r)}\int_{T(\bs,\delta)}e^{\beta\hat{H}_{N}(\bs')}\prod_{s\in\S}\frac{\omega_{N_{s}-1}}{\omega_{N_{s}}}\bigg(1-R_{s}(\bs',\bs')\bigg)^{\frac{N_{s}-2}{2}}d\rho(\bs')d\nu(\bs)
\end{align*}
where $\omega_{d}$ denotes the volume of the unit sphere in $\R^{d}$,
$\rho$ denotes the volume measure corresponding to the Riemannian
metric on $T(\bs,\delta)$ induced by the Euclidean structure in $\R^{N}$,
and 
\[
T(\bs,\delta):=\Big\{\bs'\in B(\bs_{\star},r,\delta):\,\frac{P_{s}(\bs')}{R_{s}(\bs',\bs')}=\frac{P_{s}(\bs)}{R_{s}(\bs,\bs)},\:\forall s\in\S\Big\},
\]
where $P_{s}(\bs)$ is the orthogonal projection of $(\sigma_{i})_{i\in I_{s}}$
to the orthogonal space to $(\sigma_{\star,i})_{i\in I_{s}}$.\qed

\subsection{\label{subsec:pf2}Proof of Lemma \ref{lem:Htilde_r}}

Fix some $r\in[0,1)^{\S}$. Recall the decomposition (\ref{eq:decomp})
of $H_{N}(\bs)$. In light of (\ref{eq:Hhat}), we need to prove that
\begin{equation}
\lim_{N\to\infty}\Big|\frac{1}{N}\E\log\int_{B(\bs_{\star},r)}\exp\Big(\beta\hat{H}_{N}(\bs)\Big)d\nu(\bs)-\frac{1}{N}\E\log\int_{S_{N}}\exp\Big(\beta\tilde{H}_{N}^{r}(\bs)\Big)d\mu(\bs)\Big|=0,\label{eq:phihat_codimapx-1-1}
\end{equation}
for some arbitrary $\bs_{\star}\in S_{N}$. 

For any $\bs=(\sigma_{i})_{i=1}^{N}\in B(\bs_{\star},r)$, define
$\tilde{\bs}=(\tilde{\sigma}_{i})_{i=1}^{N}\in B(\bs_{\star},0)\subset S_{N}$
by 
\begin{equation}
\tilde{\sigma}_{i}:=\sqrt{\frac{1}{1-r(s)^{2}}}(\sigma_{i}-r(s)\sigma_{\star,i}),\text{\ensuremath{\quad}if }i\in I_{s},\label{eq:sigmatilde}
\end{equation}
and define the Hamiltonian $\bar{H}_{N}(\tilde{\bs})=\hat{H}_{N}(\bs)$
on $B(\bs_{\star},0)$. Then, for any two points $\tilde{\bs}^{1},\,\tilde{\bs}^{2}$
from $B(\bs_{\star},0)$, by a straightforward calculation,
\begin{equation}
\frac{1}{N}\E\big(\bar{H}_{N}(\tilde{\bs}^{1})\bar{H}_{N}(\tilde{\bs}^{2})\big)=\tilde{\xi}_{r}(R(\tilde{\bs}^{1},\tilde{\bs}^{2}))+\xi(r^{2})-\frac{\xi(r)^{2}}{\xi(1)},\label{eq:xiqcov}
\end{equation}
where the mixture $\tilde{\xi}_{r}(x)$ is defined in (\ref{eq:xitilde_r}).

Extend the Hamiltonian $\bar{H}_{N}$ from $B(\bs_{\star},0)$ to
a centered Gaussian field on $S_{N}$ whose covariance is given by
(\ref{eq:xiqcov}). Of course,
\begin{equation}
\begin{aligned}\frac{1}{N}\E\log\int_{B(\bs_{\star},r)}\exp\Big(\beta\hat{H}_{N}(\bs)\Big)d\nu(\bs) & =\frac{1}{N}\E\log\int_{B(\bs_{\star},0)}\exp\Big(\beta\bar{H}_{N}(\bs)\Big)d\nu(\bs)\\
 & =\frac{1}{N}\E\log\int_{S_{N}}\exp\Big(\beta\bar{H}_{N}(\bs)\Big)d\mu(\bs)+o_{N}(1).
\end{aligned}
\label{eq:hatbar}
\end{equation}

Since $A(r)^{2}:=\xi(r^{2})-\frac{\xi(r)^{2}}{\xi(1)}\geq0$, we may
write the Hamiltonian $\bar{H}_{N}$ as
\[
\bar{H}_{N}(\bs)=\tilde{H}_{N}^{r}(\bs)+\sqrt{N}A(r)X,
\]
where $X$ is a standard Gaussian variable independent of the Hamiltonian
$\tilde{H}_{N}^{r}(\bs)$ with mixture $\tilde{\xi}_{r}(x)$. Obviously,
\begin{equation}
\frac{1}{N}\E\log\int_{S_{N}}\exp\Big(\beta(\tilde{H}_{N}^{r}(\bs)+\sqrt{N}A(r)X)\Big)d\mu(\bs)=\frac{1}{N}\E\log\int_{S_{N}}\exp\Big(\beta\tilde{H}_{N}^{r}(\bs)\Big)d\mu(\bs).\label{eq:HtildeF}
\end{equation}

Combining (\ref{eq:hatbar}) and (\ref{eq:HtildeF}) proves (\ref{eq:phihat_codimapx-1-1}),
and completes the proof.\qed

\subsection{\label{subsec:pf3}Proof of Lemma \ref{lem:codim_mean}}

Let $\beta\leq \beta_c$ and $r\in[0,1)^\S$. By \eqref{eq:Jensen},
\[
\frac{1}{N}\E\log\int_{S_{N}}e^{\beta\tilde{H}_{N}^{r}(\bs)}d\mu(\bs)\leq \frac 12\beta^2\tilde{\xi}_{r}(1)\leq \frac 12\beta^2 {\xi}(1),
\]
thus we only need to prove the lower bound 
\begin{equation}
	\liminf_{N\to\infty}\frac{1}{N}\E\log\int_{S_{N}}e^{\beta\tilde{H}_{N}^{r}(\bs)}d\mu(\bs)\geq \frac{1}{2}\beta^{2}{\xi}(1)-\kappa(r^{2}),\label{eq:liminfHtilde}
\end{equation}
for $\kappa(x)$ as in the statement of the lemma. 

Given $x\in[0,1)^\S$ and $t>0$, by setting $r=\sqrt{tx}$ (where the square root is applied elementwise $\sqrt{tx}(s):=\sqrt{tx(s)}$), \eqref{eq:liminfHtilde} becomes 
\begin{equation*}
	\liminf_{N\to\infty}\E \tilde F_{N,\beta}(t)\geq \frac{1}{2}\beta^{2}{\xi}(1)-\kappa(tx),
\end{equation*}
where we define
\begin{equation*}
	\tilde F_{N,\beta}(t)=\frac{1}{N}\log\int_{S_{N}}e^{\beta\tilde{H}_{N}^{\sqrt{tx}}(\bs)}d\mu(\bs).
\end{equation*}

Recall that, in distribution,
\[
\tilde{H}_{N}^{\sqrt{tx}}(\bs) = \sum_{p:|p|\geq1}  \Delta_{p,\sqrt{tx}} H_{N,p}(\bs),
\] where the coefficients $\Delta_{p,\sqrt{tx}}$ are as in \eqref{eq:xitilde_r_2} and $H_{N,p}(\bs)$ are the pure $p$-spin models with mixture $\prod_{s\in\S}x(s)^{p(s)}$ which we assume to be independent. Define
\[
\bar{H}_{N}^{t,x}(\bs) = \sum_{p:|p|\geq1}  \Delta_{p,t,x} H_{N,p}(\bs),\quad\mbox{where\ }\quad \Delta_{p,t,x}= \Delta_{p}+t\cdot \frac{d}{d\epsilon}\Big|_{\epsilon=0}\Delta_{p,\sqrt{\epsilon x}}.
\]
Since we assume that $\Delta_p$ for finitely many values of $p$, for small enough $t\geq0$,  $\Delta_{p,t,x}\geq0$ for all $p$.

Of course, for some constants $C_{p}$ and $C:=\sum_p C_{p}<\infty$  that depend on $\xi$ and $x$, for small $t$,
\[
|\Delta_{p,t,x}^2-\Delta_{p,\sqrt{tx}}^2|\leq C_{p}t^2
\]
and by Gaussian integration by parts
\[
|\bar F_{N,\beta}(t) - \tilde F_{N,\beta}(t)|\leq C t^2,
\]
where we define
\begin{equation*}
	\bar F_{N,\beta}(t)=\frac{1}{N}\log\int_{S_{N}}e^{\beta\bar{H}_{N}^{t,x}(\bs)}d\mu(\bs).
\end{equation*}

Hence, it will be enough to prove that
\begin{equation}
	\label{eq:Fbart}
	\liminf_{N\to\infty}\E \bar F_{N,\beta}(t)\geq \frac{1}{2}\beta^{2}{\xi}(1)-\kappa(tx),
\end{equation}
for an appropriate function $\kappa(x)$.

In order to be able to invoke Talagrand's positivity principle \cite{TalagBook2003,PanchenkoBook}, we will add a perturbation to the Hamiltonian.  Its definition
is taken from \cite{PanchenkoMulti} where Panchenko
introduced a multi-species version of the Ghirlanda-Guerra identities which induce the positivity of overlaps,
and showed that they are satisfied in the presence of the perturbation
Hamiltonian. The results of \cite{PanchenkoMulti} concern the multi-species SK
model, but they are general and also cover the multi-species spherical
models. The proofs for the spherical case have been worked out in
\cite{BatesSohn1}.

Let $\mathscr{W}$ be a countable dense subset of $[0,1]^{\S}$.
For any $p\geq1$ and vector 
\[
w=(w_{s})_{s\in\S}\in\mathscr{W}
\]
define $s_{i}(w)=\sqrt{w_{s}}$ for $i\in I_{s}$ and $s\in\S$,
and consider the Hamiltonian 
\[
h_{N,w,p}(\bs)=\frac{1}{N^{\frac{p}{2}}}\sum_{1\leq i_{1},\ldots,i_{p}\leq N}g_{i_{1},\ldots,i_{p}}^{w,p}\sigma_{i_{1}}s_{i_{1}}(w)\cdots\sigma_{i_{p}}s_{i_{p}}(w),
\]
where $g_{i_{1},\ldots,i_{p}}^{w,p}$ are i.i.d. standard Gaussian
variables, independent for all combinations of indices $p\geq1$ and
$1\leq i_{1},\ldots,i_{p}\leq N$. 

Consider some one-to-one function $j:\mathscr{W}\to\mathbb{N}$. Let
$y=(y_{w,p})_{w\in\mathscr{W},p\geq1}$ be i.i.d. random variables
uniform in the interval $[1,2]$ and independent of all other variables.
Define the Hamiltonian 
\[
h_{N}(\bs)=\sum_{w\in\mathscr{W}}\sum_{p\geq1}2^{-j(w)-p}y_{w.p}h_{N,w,p}(\bs).
\]
Let $\gamma$ be an arbitrary number in $(0,1/2)$ and set $s_{N}=N^{\gamma}$.
For $t\in[0,1]$, define
\begin{equation*}
	H_{N}^{t,x}(\bs)=H_{N}^{t,x}(\bs)+s_{N}h_{N}(\bs)
	\label{eq:pert}
\end{equation*}
and
\begin{equation*}
F_{N,\beta}(t)=\frac{1}{N}\log\int_{S_{N}}e^{\beta{H}_{N}^{t,x}(\bs)}d\mu(\bs).
\end{equation*}

Conditional on $y=(y_{w,p})_{w\in\mathscr{W},p\geq1}$, $h_{N}(\bs)$
is a Gaussian process with variance bounded by $4$ (see e.g. \cite[(26)]{PanchenkoMulti}). Hence, using
Jensen's inequality, we have that
\begin{equation}
	\begin{aligned}\Big| F_{N,\beta}(t) - \bar F_{N,\beta}(t)\Big| & \leq\frac{4\beta^{2}s_{N}^{2}}{N}.
	\end{aligned}
\end{equation}
Since $s_{N}^{2}/N\to0$, \eqref{eq:Fbart} will follow if we prove that
\begin{equation}
	\label{eq:Ft}
	\liminf_{N\to\infty}\E F_{N,\beta}(t)\geq \frac{1}{2}\beta^{2}{\xi}(1)-\kappa(tx),
\end{equation}

For $t=0$, $\bar F_{N,\beta}(0)=F_{N,\beta}$. Therefore, since $\beta\leq\beta_c$, \eqref{eq:Fbart} and \eqref{eq:Ft} hold with $\kappa(0)=0$.
Since $\Delta_{p,t,x}$ are affine in $t$, using H\"{o}lder's inequality one can check that $ F_{N,\beta}(t)$ is a convex function of $t$ (and thus has one-sided derivatives). Hence, in order to show that there exists a function $\kappa(x)$ which satisfies \eqref{eq:kappa0} and 	\eqref{eq:Ft}, it will be enough to show that
\begin{equation}
	\label{eq:Fbartd}
	\lim_{N\to\infty}\frac{d}{dt}^+\Big|_{t=0}\E  F_{N,\beta}(t) =  0,
\end{equation}
where $\frac{d}{dt}^+\big|_{t=0}$ denotes the derivative from the right at $0$.

Let $G_{N,\beta}^0$ be the Gibbs measure 
\begin{equation*}
	d G_{N,\beta}^0(\bs)=\frac{e^{\beta {H}_{N}^{0,x}(\bs)}}{\int_{S_{N}}e^{\beta {H}_{N}^{0,x}(\bs')}d\mu(\bs')}d\mu(\bs)
\end{equation*}
corresponding to 
\begin{equation}
\label{eq:H0}
{H}_{N}^{0,x}(\bs)\overset{d}{=} H_N(\bs) +s_Nh_N(\bs).	
\end{equation}
For any function $f(\bs^{1},\ldots,\bs^{n})$ of $n$ points from
$S_{N}$ denote 
\[
\langle f(\bs^{1},\ldots,\bs^{n})\rangle =\int_{S_{N}^{n}}f(\bs^{1},\ldots,\bs^{n})d( G_{N,\beta}^0)^{\otimes n},
\]
where $(G_{N,\beta}^0)^{\otimes n}$ denotes the $n$-fold product
measure of $G^0_{N,\beta}$ with itself.

Note that
\begin{equation}
	\frac{d}{dt}^+\Big|_{t=0}\E  F_{N,\beta}(t)=\frac{1}{N}\E\Big\langle
	\beta \sum_{p:|p|\geq1} \frac{d}{d\epsilon}\Big|_{\epsilon=0}\Delta_{p,\sqrt{\epsilon x}}\cdot H_{N,p}(\bs)
	\Big\rangle.\label{eq:ddtFt}
\end{equation}

By Gaussian integration by parts \cite[Lemma 1.1]{PanchenkoBook}, the right-hand side of \eqref{eq:ddtFt} is equal to 
\begin{equation}
	\E\,\big\langle C(\bs^{1},\bs^{1})- C(\bs^{1},\bs^{2})\big\rangle,
\end{equation}
where 
\begin{align*}
C(\bs^{1},\bs^{2})&:=\frac{\beta^2}{N}\sum_{p:|p|\geq1} \Big(\Delta_{p}\frac{d}{d\epsilon}\Big|_{\epsilon=0}\Delta_{p,\sqrt{\epsilon x}}\Big)\cdot\E\Big(
 H_{N,p}(\bs^1)H_{N,p}(\bs^2)
\Big)\\
&=\frac12\beta^2 \eta_x(R(\bs^{1},\bs^{2}))
,
\end{align*}
where we denote $\partial_s\xi(r)=\frac{d}{dr(s)}\xi(r)$ and for $z:[-1,1]^\S\to\R$,
\[
\eta_x(z):=\frac{d}{d\epsilon}\Big|_{\epsilon=0} \tilde\xi_{\sqrt{\epsilon x}}(z)= \sum_{s\in\S}\Big(
\partial_s\xi(z)x(s)(1-z(s))-\partial_s\xi(0)x(s)
\Big).
\]

Combining the above, we obtain that
\begin{equation}
	\lim_{N\to\infty}\frac{d}{dt}^+\Big|_{t=0}\E  F_{N,\beta}(t) = \frac12\beta^2 \Big( 
	\eta_x(1)-\lim_{N\to\infty}\E \big\langle \eta_x(R(\bs^1,\bs^2))\big\rangle
	\Big).
	\label{eq:ddt+}
\end{equation}

Note that since $\partial_s\xi(0)=0$,
\begin{equation}
\eta_x(1) = \eta_x(0) = 0.
\label{eq:eta01}
\end{equation}
Hence, to complete the proof of Lemma \ref{lem:codim_mean} it will be enough to show that for any $\epsilon>0$,
\begin{equation}
\lim_{N\to\infty}\E \Big\langle 
\mathbf{1}\big\{\max_{s\in\S}|R_{s}(\bs^{1},\bs^{2})|<\epsilon\big\}
\Big\rangle=0,\label{eq:R0}
\end{equation}
where $\mathbf{1}\{A\}$ is the indicator of an event $A$.

Note that since $\beta\leq \beta_c$,\footnote{For $\beta=\beta_c$ the derivative from the left of $\lim_{N\to\infty}\E  F_{N,\beta}$ is $\beta_c\xi(1)$. Since $\lim_{N\to\infty}\E  F_{N,\beta}$ is convex in $\beta$, its derivative from the right exists and is bounded from below by the derivative from the left. The latter is bounded by $\beta_c\xi(1)$, by \eqref{eq:Jensen}. Hence the derivative at $\beta_c$ exists and is equal to $\beta_c\xi(1)$.}
\begin{align*}
\beta\xi(1)&=\frac{d}{d\beta}\lim_{N\to\infty}\E  F_{N,\beta}(0) =\lim_{N\to\infty}\frac{d}{d\beta}\E  F_{N,\beta}(0) = 
\lim_{N\to\infty}\frac{1}{N}\E\big\langle {H}_{N}^{0,x}(\bs)
\big\rangle\\ 
&=\beta\Big(
\xi(1) - \lim_{N\to\infty}\E\big\langle \xi(R(\bs^1,\bs^2))\big\rangle
\Big),
\end{align*}
where the limit and derivative may be interchanged by the convexity of $\beta\mapsto \E  F_{N,\beta}(0)=\E  F_{N,\beta}$, and the last equality follows from Gaussian integration by parts. 

Thanks to the perturbation term $s_Nh_N(\bs)$ in \eqref{eq:pert}, from Lemma 3.3 of \cite{BatesSohn1}, for any $\epsilon>0$, 
\begin{equation*}
	\lim_{N\to\infty}\E\Big\langle\mathbf{1}\big\{\min_{s\in\S}R_{s}(\bs^{1},\bs^{2})<-\epsilon\big\}\Big\rangle=0.
\end{equation*}

Combining the above with the assumption \Asm{} one easily concludes \eqref{eq:R0}, which completes the proof.\qed

\section*{Appendix}
In this appendix we explain how the following result follows by slightly modifying our proofs. Here we denote by $\langle\cdot\rangle$ averaging by the Gibbs measure $G_{N,\beta}$ corresponding directly to $H_N(\bs)$,
\begin{equation*}
	d G_{N,\beta}(\bs)=\frac{e^{\beta {H}_{N}(\bs)}}{\int_{S_{N}}e^{\beta {H}_{N}(\bs')}d\mu(\bs')}d\mu(\bs).
\end{equation*}

\begin{lem}
	\label{lem:appendix}
	Let $H_N(\bs)$ be a model corresponding to a mixture $\xi(x)=\sum_{p\in P}\Delta_p^2\prod_{s\in\S}x(s)^{p(s)}$ such that  $\Delta_p\neq0$ for finitely many $p$ and let $\beta\leq\beta_c$. Suppose that for any $\epsilon>0$,
	\begin{equation}
		\lim_{N\to\infty}\E \Big\langle 
		\mathbf{1}\big\{\max_{s\in\S}|R_{s}(\bs^{1},\bs^{2})|<\epsilon\big\}
		\Big\rangle=0.\label{eq:R0-2}
	\end{equation}
Then, the matrix 
\begin{equation}
	\label{mat}\Big(\frac{d}{dr(s)}\frac{d}{dr(t)}f_{\beta}(0)\Big)_{s,t\in\S}
\end{equation} is negative semi-definite.
\end{lem}

Note that above instead of assuming \Asm{} as we did in the main results, we assume \eqref{eq:R0-2}. This result is crucial to the analysis of the TAP representation of the multi-species pure $p$-spin models in \cite{TAPmulti2}, where indeed we need to deal with a model which for some values of $p$  may not satisfy \Asm{}.

First we explain how the conclusion of Proposition \ref{prop:main} follows in the setting of Lemma \ref{lem:appendix}.
The only place we used the assumption of \Asm{} in the proof of Proposition \ref{prop:main} is the very last step in the proof of Lemma \ref{lem:codim_mean}, to prove \eqref{eq:R0}. 
So we only need to explain how to prove the conclusion \eqref{eq:phihat_codimapx-1} of Lemma \ref{lem:codim_mean}. By the same argument as in the proof of the latter lemma, the lemma  follows if we can prove \eqref{eq:Fbart}. By the argument we used around \eqref{eq:Fbartd}, to prove \eqref{eq:Fbart} it is enough to show that 
\begin{equation}
	\label{eq:Fbartd00}
	\lim_{N\to\infty}\frac{d}{dt}^+\Big|_{t=0}\E  \bar F_{N,\beta}(t) =  0.
\end{equation}

As in \eqref{eq:ddt+}, 
\begin{equation}
	\lim_{N\to\infty}\frac{d}{dt}^+\Big|_{t=0}\E  \bar F_{N,\beta}(t) = \frac12\beta^2 \Big( 
	\eta_x(1)-\lim_{N\to\infty}\E \big\langle \eta_x(R(\bs^1,\bs^2))\big\rangle
	\Big),
\end{equation}
where now the averaging is w.r.t. the Gibbs measure corresponding to the Hamiltonian without perturbation, which have the same law as $G_{N,\beta}$. Thus, \eqref{eq:Fbartd00} follows from \eqref{eq:R0-2} and \eqref{eq:eta01}. This proves the conclusion of Proposition \ref{prop:main} in the setting of Lemma \ref{lem:appendix}.

Now assume towards contradiction that the matrix \eqref{mat} has a positive eigenvalue. 
Then for some $r$, which we can choose to be in $[0,1)^\S$,
\[
\frac{d}{d\alpha}\Big|_{\alpha=0}f_{\beta}(\alpha r)=0\text{\ensuremath{\quad and\quad}}\frac{d^{2}}{d\alpha^{2}}\Big|_{\alpha=0}f_{\beta}(\alpha r)>0.
\]
By the same argument as in the proof of Proposition \ref{prop:singular}, this leads to a contradiction to the fact that $\beta\leq \beta_c$, from which we conclude that
  \eqref{mat} is negative  semi-definite.

\bibliographystyle{plain}
\bibliography{master}

\end{document}